\documentclass[12pt]{amsart}
\usepackage{amssymb,latexsym}
\usepackage{enumerate}

\makeatletter
\@namedef{subjclassname@2010}{%
  \textup{2010} Mathematics Subject Classification}
\makeatother

\DeclareMathOperator{\RE}{Re}

\newtheorem{thm}{Theorem}[section]

\newtheorem{lem}[thm]{Lemma}

\theoremstyle{definition}
\newtheorem{defin}[thm]{Definition}
\newtheorem{rem}[thm]{Remark}

\numberwithin{equation}{section}

\begin{document}





\title[Differential Subordination]{Applications
of Theory of Differential Subordination for Functions with Fixed Initial Coefficient to Univalent Functions}

\author[S. Nagpal]{Sumit Nagpal}
\address{Department of Mathematics, University of Delhi,
Delhi--110 007, India}

\email{sumitnagpal.du@gmail.com }

\author{V. Ravichandran}

\address{Department of Mathematics, University of Delhi,
Delhi--110 007, India \and School of Mathematical Sciences,
Universiti Sains Malaysia, 11800 USM, Penang, Malaysia}

\email{vravi@maths.du.ac.in; vravi68@gmail.com}

\date{}

\begin{abstract}
By using the theory of first-order differential subordination  for functions with fixed initial coefficient, several well-known results for subclasses of univalent functions are improved by restricting the functions to have fixed second coefficient. The influence of the second coefficient of univalent functions is evident in the results obtained.
\end{abstract}



\maketitle

\section{Introduction and preliminaries}\label{sec1}
It is well-known that the second coefficient of univalent functions influences many properties. For example, the bound for the second coefficient of univalent functions yields the growth and distortion estimates  as well the  Koebe constant.  Various subclasses of univalent functions with fixed second coefficients were investigated beginning with Gronwall \cite{gronwall}.  For a brief survey of these developments as well as for some radius problem, see \cite{naveen}. The necessary modifications to the theory of differential subordination to handle problems for functions with second coefficients are recently carried out in \cite{sumit}. Using the results in  \cite{sumit},   the influence of the second coefficient in certain differential implications associated with starlike and convex functions with fixed second coefficients  is investigated in this paper.

Let $p$ be an analytic function in the unit disk $\mathbb{D}=\{z \in \mathbb{C}:|z|<1\}$ and $\psi(r,s)$ be a complex function defined in a
domain of $\mathbb{C}^{2}$. Consider a class of functions $\Psi$, and two subsets $\Omega$ and $\Delta$ in $\mathbb{C}$. Given any two
quantities, the aim of the theory of first-order differential subordination is to determine the third so that the following differential implication is satisfied:
\begin{equation*}\label{subo}
\psi \in \Psi \quad \mbox{and} \quad \{\psi(p(z),z p'(z)): z \in \mathbb{D}\} \subset \Omega \quad \Rightarrow \quad p(\mathbb{D}) \subset \Delta.
\end{equation*}
Furthermore, the problem is to find ``smallest'' such $\Delta$ and ``largest'' such $\Omega$.
In \cite{sumit}, the authors proposed a new methodology by making appropriate modifications and improvements to the Miller and Mocanu's theory (see \cite{millermocanu1978,millermocanu1981} and their monograph \cite{monograph}) of second-order differential subordination and gave interesting applications of the newly formulated theory to the classes of normalized convex and starlike functions with fixed second coefficient. Let $\mathcal{H}_{\beta}[a,n]$ consist  of analytic functions $p$ of the form
\[p(z)=a+\beta z^{n}+p_{n+1}z^{n+1}+\cdots,\]
where $\beta \in \mathbb{C}$ is fixed. Without loss of generality, we assume that $\beta$ is a positive real number.
\begin{defin}\label{sec1,def1}\cite[Definition 1, p.\ 158]{millermocanu1981}
Let $Q$ be the class of functions $q$ that are analytic and injective in
$\overline{\mathbb{D}}\setminus E(q)$ where \[ E(q) := \{ \zeta \in
\partial \mathbb{D}: \lim_{z \rightarrow \zeta}
q(z)=\infty\}\] and are such that $q'(\zeta) \neq 0 $ for $\zeta\in
\partial \mathbb{D}\setminus E(q)$.
\end{defin}

\begin{defin}\label{sec1,def2}\cite[Definition 3.1, p. 616]{sumit}
Let $\Omega$ be a domain in $\mathbb{C}$, $n \in \mathbb{N}$ and $\beta >0$. Let $q \in {Q}$ be such that $ |q'(0)|\geq \beta$. The class
$\Psi_{n,\beta}(\Omega,q)$ consists of  \emph{$\beta$-admissible functions} $\, \psi:\mathbb{C}^2\rightarrow \mathbb{C}$   satisfying the following
conditions:
\begin{itemize}
  \item[(i)] $ \psi(r,s)$ is continuous in a domain $D\subset \mathbb{C}^2$,
  \item[(ii)] $(q(0),0)\in{D}$ and $\psi(q(0),0)\in{\Omega}$,
  \item[(iii)] $\psi(q(\zeta),m \zeta  q'(\zeta))\not\in{\Omega}$ whenever $(q(\zeta),m \zeta  q'(\zeta))\in{D}$, $\zeta\in
\partial \mathbb{D}\setminus E(q)$ and
  \[m \geq n+\frac{|q'(0)|-\beta }{|q'(0)|+\beta }.\]
\end{itemize}
We write $\Psi_{1,\beta}(\Omega,q)$ as $\Psi_{\beta}(\Omega,q)$.
\end{defin}

\begin{thm}\label{sec1,th3}\cite[Theorem 3.1, p. 617]{sumit}
Let $q(0)=a$, $\psi \in \Psi_{n,\beta}(\Omega,q)$ with associated
domain D, and $\beta >0$ with $ |q'(0)|\geq \beta$.
Let $p \in \mathcal{H}_{\beta}[a,n]$. If $(p(z), zp'(z))\in {D}$ for $z \in \mathbb{D}$ and
\[ \psi (p(z), zp'(z)) \in {\Omega}\quad ( z\in
 {\mathbb{D}})\]
 then $p\prec q$.
\end{thm}

The special case of $\Delta$ being a half plane is important in our investigation. Let $\Delta=\{w:\RE w>0\}$.  The function
\[q(z)=\frac{a+\overline{a}z}{1-z} \quad (z \in \mathbb{D})\]
where $\RE a >0$, is univalent in $\overline{\mathbb{D}}\setminus\{1\}$ and satisfies $q(\mathbb{D})=\Delta$, $q(0)=a$ and $q \in Q$. Let
$\Psi_{n,\beta}(\Omega,a):=\Psi_{n,\beta}(\Omega,q)$ and when $\Omega=\Delta$, denote the class by $\Psi_{n,\beta}\{a\}$ with $\Psi_{\beta}\{a\}:=\Psi_{1,\beta}\{a\}$.
The class $\Psi_{n,\beta}(\Omega,a)$ consists of those functions $\psi:\mathbb{C}^2\rightarrow \mathbb{C}$ that are continuous in a domain $D \subset \mathbb{C}^2$ with $(a,0) \in D$ and $\psi(a,0) \in \Omega$, and that satisfy the admissibility condition:
\begin{equation}\label{sec1,eq1}
\begin{split}
\psi(i\rho,\sigma) &\not\in \Omega \quad \mbox{whenever} \quad (i\rho,\sigma) \in D \quad \mbox{and}\quad\\
  \sigma &\leq-\frac{1}{2}\left( n+\frac{2\RE a-\beta }{2\RE a+\beta }\right)\frac{|a-i\rho|^{2}}{\RE a},
\end{split}
\end{equation}
where $\rho \in \mathbb{R}$ and $n \geq 1$.

If $a=1$, then \eqref{sec1,eq1} simplifies to
\begin{equation}\label{sec1,eq2}
\begin{split}
\psi(i\rho,\sigma) &\not\in \Omega \quad \mbox{whenever} \quad (i\rho,\sigma) \in D \quad \mbox{and}\quad\\
                                \sigma &\leq-\frac{1}{2}\left( n+\frac{2-\beta }{2+\beta }\right)
                                 (1+\rho^{2}),
\end{split}
\end{equation}
where $\rho \in \mathbb{R}$, and $n \geq 1$.

In this particular case, Theorem \ref{sec1,th3} becomes
\begin{thm}\label{sec1,th4} \cite[Theorem 3.4, p. 620]{sumit}
Let $p \in \mathcal{H}_{\beta}[a,n]$ with $\RE a>0$ and $0<\beta  \leq
2\RE a$.
\begin{enumerate}
  \item [(i)] Let $\psi \in \Psi_{n,\beta}(\Omega,a)$ with associated domain $D$. If $(p(z), zp'(z))\in {D}  $ and $\psi (p(z), zp'(z)) \in \Omega\quad ( z\in {\mathbb{D}})$,
              then $\RE p(z) >0\quad (z \in \mathbb{D})$.
  \item [(ii)] Let $\psi \in \Psi_{n,\beta}\{a\}$ with associated domain $D$. If $(p(z), zp'(z))\in {D}$ and $ \RE \psi (p(z), zp'(z))>0 \quad ( z\in {\mathbb{D}})$,
              then $\RE p(z)>0\quad (z \in \mathbb{D})$.
\end{enumerate}
\end{thm}

\section{Applications in univalent function theory}\label{sec2}
Let $\mathcal{A}_n$ be the class consisting of analytic functions $f$ defined on
$\mathbb{D}$ of the form $f(z)=z+a_{n+1}z^{n+1}+a_{n+2}z^{n+2}+\cdots$, and $\mathcal{A}:=\mathcal{A}_1$. The class $\mathcal{S}^*(\alpha)$ of
starlike functions of order $\alpha$, $0 \leq \alpha <1$, consists of functions $f\in \mathcal{A}$ satisfying the inequality
\[ \RE\left(\frac{zf'(z)}{f(z)}\right)>\alpha \quad ( z\in
 {\mathbb{D}}). \] Similarly,  the class
$\mathcal{C}(\alpha)$ of convex functions of order $\alpha$, $0 \leq
\alpha <1$, consists of functions $f\in \mathcal{A}$ satisfying the
inequality
\[ \RE\left(1+\frac{zf''(z)}{f'(z)}\right)>\alpha \quad ( z\in
 {\mathbb{D}}). \]
When $\alpha=0$, these classes are respectively denoted by $\mathcal{S}^*$ and $\mathcal{C}$. Let $\mathcal{A}_{n,b}$ denote the class of
functions $f \in \mathcal{A}_{n}$ of  the form
\[f(z)=z+bz^{n+1}+a_{n+2}z^{n+2}+\cdots,\]
where $b$ is fixed. We write $\mathcal{A}_{1,b}$
as $\mathcal{A}_{b}$.

There are many differential inequalities in classical analysis for which the differential operator is required to have positive real part. A typical example is the Marx-Strohh\"{a}cker result, which states that if $ f \in \mathcal{A}$, then
\[\RE \left(\frac{z f''(z)}{f'(z)}+1\right) >0 \quad(z\in \mathbb{D})\quad \Rightarrow  \quad \RE \frac{z f'(z)}{f(z)}>\frac{1}{2}\quad (z\in \mathbb{D}).\]
A natural problem is to extend the result by finding a domain $D$ containing the right half-plane so that
\[\frac{z f''(z)}{f'(z)}+1 \in D  \quad(z\in \mathbb{D})\quad \Rightarrow \quad \RE \frac{z f'(z)}{f(z)}>\frac{1}{2} \quad(z\in \mathbb{D}).\]
 The domain $D$ cannot be taken as the half-plane $\{w\in\mathbb{C}:\RE w>\alpha\}$, with  $ \alpha<0$,  for functions $f \in \mathcal{A}$. (For a counter example, see \cite{example}.) However, it is possible to take such a $D$ for functions $f \in \mathcal{A}_{b}$. To prove this result, we shall need the following lemma proved by Ozaki.

\begin{lem}\label{lem} \cite{ozaki}
If $f \in \mathcal{A}$ satisfies
\[\RE \left(\frac{z f''(z)}{f'(z)}+1\right) >-\frac{1}{2}\quad (z\in \mathbb{D}),\]
then $f$ is univalent in $\mathbb{D}$.
\end{lem}

\begin{thm}\label{sec2,th2}
If $f \in \mathcal{A}_{b}$ with $|b| \leq 1$, then the following implication holds:
\[\RE \left(\frac{z f''(z)}{f'(z)}+1\right) >\frac{|b|-1}{2(|b|+1)} \quad (z\in \mathbb{D})\quad \Rightarrow \quad \RE \frac{z f'(z)}{f(z)}>\frac{1}{2}\quad (z\in \mathbb{D}).\]
\end{thm}

\begin{proof}
If we set
\begin{equation}\label{sec2,eq1}
\alpha:=\frac{|b|-1}{2(|b|+1)}
\end{equation}
then $\alpha \in [-1/2,0]$.
Define the function $p:\mathbb{D} \rightarrow  \mathbb{C}$ by
\[p(z):=2\frac{zf'(z)}{f(z)}-1.\]
Using Lemma \ref{lem}, it follows that $f$ is univalent and hence
\[p(z)=1+2bz+\cdots\]
is analytic in $\mathbb{D}$. Thus $p \in \mathcal{H}_{2b}[1,1]$ and satisfies
\begin{equation}\label{sec2,eq2}
\frac{zf''(z)}{f'(z)}+1-\alpha=\frac{p(z)+1}{2}+\frac{zp'(z)}{p(z)+1}-\alpha=\psi(p(z),zp'(z))\quad (z \in \mathbb{D})
\end{equation}
where
\[\psi(r,s):=\frac{r+1}{2}+\frac{s}{r+1}-\alpha,\]
and $\alpha$ is given by \eqref{sec2,eq1}.
The function $\psi$ is continuous in the domain $D=(\mathbb{C}\backslash \{-1\})\times \mathbb{C}$, $(1,0) \in D$ and
\[\RE \psi(1,0)=1-\alpha >0,\]
as $\alpha \in [-1/2,0]$. We need to show that the admissibility condition \eqref{sec1,eq2} is satisfied. Since
\[\psi(i\rho,\sigma)=\frac{i\rho+1}{2}+\frac{\sigma}{1+\rho^{2}}(1-i\rho)-\alpha\]
we have
\begin{align*}
\RE \psi(i\rho,\sigma)&=\frac{1}{2}+\frac{\sigma}{1+\rho^{2}}-\alpha\\
                      & \leq \frac{1}{2}-\frac{1}{2}\left(1+\frac{2-2|b|}{2+2|b|}\right)-\alpha\\
                      &=\frac{|b|-1}{2(|b|+1)}-\alpha=0,
\end{align*}
whenever $\rho \in \mathbb{R}$ and
\[\sigma \leq -\frac{1}{2}\left(1+\frac{2-\beta}{2+\beta}\right)(1+\rho^{2}),\quad \beta=2|b|.\]
Thus $\psi \in \Psi_{2b}\{1\}$.

From the hypothesis and \eqref{sec2,eq2}, we obtain
\[\RE \psi(p(z),zp'(z)) >0 \quad ( z \in \mathbb{D}).\]
Therefore, by applying     Theorem \ref{sec1,th4} (ii), we conclude that $p$ satisfies
\[\RE p(z) >0 \quad ( z \in \mathbb{D}).\]
This is equivalent to
\[\RE \frac{z f'(z)}{f(z)}>\frac{1}{2}\quad (z\in \mathbb{D}).\qedhere\]
\end{proof}

\begin{rem}\label{sec2,rem3} For $|b|=1$,  Theorem \ref{sec2,th2} reduces to  \cite[ Theorem 2.6a, p.\ 57]{monograph}.
Also, if $|b|=0$ then $f \in \mathcal{A}_{2}$ and $\alpha=-1/2$. Therefore, Theorem \ref{sec2,th2} reduces to   \cite[ Theorem 2.6i, p.\ 68]{monograph} in this case.
\end{rem}


\begin{thm}\label{sec2,th5}
If $f \in \mathcal{A}_{b}$ with $|b| \leq 1$, then the following implication holds:
\[\RE \left(\frac{z f''(z)}{f'(z)}+1\right) >\frac{|b|-1}{|b|+1}\quad (z\in \mathbb{D}) \quad \Rightarrow \quad \RE \sqrt{f'(z)}>\frac{1}{2}\quad (z\in \mathbb{D}),\]
where the branch of the square root is so chosen that $\sqrt{1}=1$.
\end{thm}

\begin{proof}Set
\begin{equation}\label{sec2,eq3}
\alpha:=\frac{|b|-1}{|b|+1}.
\end{equation}
Then  $\alpha \in [-1,0]$.   Define the function $p:\mathbb{D} \rightarrow  \mathbb{C}$ by
\[p(z):=2\sqrt{f'(z)}-1.\]
Using the hypothesis, it follows that the function
\[p(z)=1+2bz+\cdots\]
is analytic in $\mathbb{D}$. Thus $p \in \mathcal{H}_{2b}[1,1]$ and satisfies
\begin{equation}\label{sec2,eq4}
\frac{zf''(z)}{f'(z)}+1-\alpha=1+\frac{2zp'(z)}{p(z)+1}-\alpha=\psi(p(z),zp'(z))\quad (z \in \mathbb{D})
\end{equation}
where
\[\psi(r,s):=1+\frac{2s}{r+1}-\alpha\]
and $\alpha$ is given by \eqref{sec2,eq3}. The function $\psi$ is continuous in the domain $D=(\mathbb{C}\backslash \{-1\})\times \mathbb{C}$, $(1,0) \in D$ and
\[\RE \psi(1,0)=1-\alpha >0,\]
as $\alpha \in [-1,0]$. We now show that the admissibility condition \eqref{sec1,eq2} is satisfied. Since
\[\psi(i\rho,\sigma)=1+\frac{2\sigma}{1+\rho^{2}}(1-i\rho)-\alpha\]
we have
\begin{align*}
\RE \psi(i\rho,\sigma)&=1+\frac{2\sigma}{1+\rho^{2}}-\alpha\\
                      & \leq 1-\left(1+\frac{2-2|b|}{2+2|b|}\right)-\alpha\\
                      & =\frac{|b|-1}{|b|+1}-\alpha=0,
\end{align*}
whenever $\rho \in \mathbb{R}$ and
\[\sigma \leq -\frac{1}{2}\left(1+\frac{2-\beta}{2+\beta}\right)(1+\rho^{2}), \quad \beta=2|b|.\]
Thus $\psi \in \Psi_{2b}\{1\}$.

From the hypothesis and \eqref{sec2,eq4}, we obtain
\[\RE \psi(p(z),zp'(z)) >0 \quad ( z \in \mathbb{D}).\]
Therefore, by applying  Theorem \ref{sec1,th4} (ii), we conclude that $p$ satisfies
\[\RE p(z) >0 \quad ( z \in \mathbb{D}).\]
This is equivalent to
\[\RE \sqrt{f'(z)} >\frac{1}{2}\quad (z\in \mathbb{D}).\qedhere\]
\end{proof}

\begin{rem}\label{sec2,rem6}
If $|b|=1$, then $\alpha=0$ and  Theorem \ref{sec2,th5} reduces to \cite[Theorem 2.6a, p.\ 57]{monograph}.
\end{rem}


\begin{thm}\label{sec2,th7}
If $f \in \mathcal{A}_{b}$ with $|b| \leq 1$, then the following implication holds:
\[\RE \frac{zf'(z)}{f(z)} >\frac{|b|}{|b|+1}\quad (z\in \mathbb{D}) \quad \Rightarrow \quad \RE \frac{f(z)}{z}>\frac{1}{2}\quad (z\in \mathbb{D}).\]
\end{thm}

\begin{proof}
Setting
\begin{equation}\label{sec2,eq5}
\alpha:=\frac{|b|}{|b|+1},
\end{equation}
it is seen that  $\alpha \in [0,1/2]$. Define the function $p:\mathbb{D} \rightarrow  \mathbb{C}$ by
\[p(z):=2\frac{f(z)}{z}-1.\]
Since $f \in \mathcal{A}_{b}$, the function
\[p(z)=1+2bz+\cdots\]
is analytic in $\mathbb{D}$. Thus $p \in \mathcal{H}_{2b}[1,1]$ and satisfies
\begin{equation}\label{sec2,eq6}
\frac{zf'(z)}{f(z)}-\alpha=1+\frac{zp'(z)}{p(z)+1}-\alpha=\psi(p(z),zp'(z))\quad (z \in \mathbb{D})
\end{equation}
where
\[\psi(r,s):=1+\frac{s}{r+1}-\alpha,\]
and $\alpha$ is given by \eqref{sec2,eq5}.
The function $\psi$ is continuous in the domain $D=(\mathbb{C}\backslash \{-1\})\times \mathbb{C}$, $(1,0) \in D$ and
\[\RE \psi(1,0)=1-\alpha >0,\]
as $\alpha \in [0,1/2]$. We now show that the admissibility condition \eqref{sec1,eq2} is satisfied. Since
\[\psi(i\rho,\sigma)=1+\frac{\sigma}{1+\rho^{2}}(1-i\rho)-\alpha\]
we have
\begin{align*}
\RE \psi(i\rho,\sigma)&=1+\frac{\sigma}{1+\rho^{2}}-\alpha\\
                      & \leq 1-\frac{1}{2}\left(1+\frac{2-2|b|}{2+2|b|}\right)-\alpha\\
                      & =\frac{|b|}{|b|+1}-\alpha=0,
\end{align*}
whenever $\rho \in \mathbb{R}$ and
 \[\sigma \leq -\frac{1}{2}\left(1+\frac{2-\beta}{2+\beta}\right)(1+\rho^{2}),\quad \beta=2|b|.\]
Thus $\psi \in \Psi_{2b}\{1\}$.

From the hypothesis and \eqref{sec2,eq6}, we obtain
\[\RE \psi(p(z),zp'(z)) >0 \quad( z \in \mathbb{D}).\]
Therefore, by applying  Theorem \ref{sec1,th4} (ii),  we conclude that $p$ satisfies
\[\RE p(z) >0 \quad (z \in \mathbb{D}).\]
This is equivalent to
\[\RE \frac{f(z)}{z} >\frac{1}{2}\quad (z\in \mathbb{D}).\qedhere\]
\end{proof}

\begin{rem}\label{sec2,rem8}
If $|b|=1$ then $\alpha=1/2$ and Theorem \ref{sec2,th7} reduces to \cite[Theorem 2.6a, p.\ 57]{monograph}.
\end{rem}


\begin{thm}\label{sec2,th9}
If $f \in \mathcal{A}_{b}$ is locally univalent with $|b| \leq 1$, then the following implication holds:
\[\RE \sqrt{f'(z)}>\sqrt{\frac{1+|b|}{8}} \quad (z\in \mathbb{D})\quad \Rightarrow \quad \RE \frac{f(z)}{z} >\frac{1}{2}\quad (z\in \mathbb{D}),\]
where the branch of the square root is so chosen that $\sqrt{1}=1$.
\end{thm}

\begin{proof}
To begin with, note that if we set
\begin{equation}\label{sec2,eq7}
\alpha:=\sqrt{\frac{1+|b|}{8}},
\end{equation}
then $\alpha \in [1/2\sqrt{2},1/2]$.
Define the function $p:\mathbb{D} \rightarrow  \mathbb{C}$ by
\[p(z):=\frac{2f(z)}{z}-1.\]
Since $f \in \mathcal{A}_{b}$, the function
\[p(z)=1+2bz+\cdots\]
is analytic in $\mathbb{D}$. Thus $p \in \mathcal{H}_{2b}[1,1]$ and satisfies
\begin{equation}\label{sec2,eq8}
\sqrt{f'(z)}-\alpha=\sqrt{\frac{zp'(z)+p(z)+1}{2}}-\alpha=\psi(p(z),zp'(z))\quad (z \in \mathbb{D})
\end{equation}
where
\[\psi(r,s):=\sqrt{\frac{r+s+1}{2}}-\alpha,\]
and $\alpha$ is given by \eqref{sec2,eq7}. The function $\psi$ is continuous in the domain $D=\mathbb{C}^{2}$, $(1,0) \in D$ and
\[\RE \psi(1,0)=1-\alpha >0,\]
as $\alpha \in [1/2\sqrt{2},1/2]$. We now show that the following admissibility condition holds:
\begin{equation}\label{sec2,eq9}
\RE \psi(i\rho,\sigma)=\RE \sqrt{\frac{i\rho+\sigma+1}{2}}-\alpha \leq 0,
\end{equation}
whenever $\rho \in \mathbb{R}$ and
\[\sigma \leq -\frac{1}{2}\left(1+\frac{2-\beta}{2+\beta}\right)(1+\rho^{2}),\quad \beta=2|b|.\]
If we let $\zeta=\xi+i\eta=(1+\sigma+i\rho)/2$, and using the conditions on $\rho$ and $\sigma$, we obtain
  \begin{align*}
   \xi=\frac{1+\sigma}{2} &\leq \frac{1}{2}\left[1-\frac{1}{2}\left(1+\frac{2-2|b|}{2+2|b|}\right)(1+\rho^{2})\right]\\
                          &=\frac{1}{2}\left[1-\frac{1}{1+|b|}(1+\rho^{2})\right]\\
                          &=\frac{1}{2(1+|b|)}(|b|-4\eta^{2}).
  \end{align*}
  This implies that $\zeta$ is a point inside the parabola
  \[\eta^{2}=-\frac{(1+|b|)}{2}\left[\xi-\frac{|b|}{2(1+|b|)}\right]\]
  and
  \[\RE \sqrt{\zeta}=\RE \sqrt{\xi+i\eta}=\sqrt{\frac{\xi+\sqrt{\xi^{2}+\eta^{2}}}{2}}.\]
  Since
 \begin{align*}
  \xi^{2}+\eta^{2} &\leq \frac{1}{4(1+|b|)^{2}}(|b|-4\eta^{2})^{2}+\eta^{2}\\
                   &=\frac{1}{4(1+|b|)^{2}}(1+4\eta^{2})(|b|^{2}+4\eta^{2})
 \end{align*}
  and using the fact that the geometric mean is less than or equal to the arithmetic mean, we have
  \begin{align*}
  \sqrt{ \xi^{2}+\eta^{2}}  &=\frac{1}{2(1+|b|)}\sqrt{(1+4\eta^{2})(|b|^{2}+4\eta^{2})}\\
                                &\leq \frac{1}{4(1+|b|)}[1+8\eta^{2}+|b|^{2}]
  \end{align*}
  so that
  \begin{align*}
  \xi+\sqrt{ \xi^{2}+\eta^{2}} &\leq \frac{1}{2(1+|b|)}(|b|-4\eta^{2})+\frac{1}{4(1+|b|)}[1+8\eta^{2}+|b|^{2}]\\
                               &=\frac{1+|b|}{4}.
  \end{align*}
  Thus
  \[\RE \sqrt{\zeta}- \sqrt{\frac{1+|b|}{8}} \leq 0.\]
  This is exactly the admissibility condition given in \eqref{sec2,eq9}.
Thus $\psi \in \Psi_{2b}\{1\}$.

From the hypothesis and \eqref{sec2,eq8}, we obtain
\[\RE \psi(p(z),zp'(z)) >0 \quad(z \in \mathbb{D}).\]
Therefore, by applying part $(ii)$ of Theorem \ref{sec1,th4} we conclude that $p$ satisfies
\[\RE p(z) >0 \quad ( z \in \mathbb{D}).\]
This is equivalent to
\[\RE \frac{f(z)}{z} >\frac{1}{2}\quad (z\in \mathbb{D}).\qedhere\]
\end{proof}

\begin{rem}\label{sec2,rem10}
If $|b|=1$, then $\alpha=1/2$ and  Theorem \ref{sec2,th9} reduces to \cite[Theorem 2.6a, p.\ 57]{monograph}.
\end{rem}


\section{Two Sufficient conditions for Starlikeness}\label{sec3}
In 1989, Nunokawa \cite{nunokawa} gave the following sufficient condition for  starlikeness: if $f \in \mathcal{A}$, then
\[\RE\left( \frac{zf''(z)}{f'(z)}+1\right)< \frac{3}{2}\quad (z\in \mathbb{D}) \quad \Rightarrow \quad 0<\RE \frac{zf'(z)}{f(z)}<\frac{4}{3}\quad (z\in \mathbb{D}).\]
We will improve this result for a function $f \in \mathcal{A}_{b}$.

\begin{thm}\label{sec3,th1}
If $f \in \mathcal{A}_{b}$ satisfies
\[\RE \left(\frac{zf''(z)}{f'(z)}+1\right) <\frac{3}{2}\quad (z\in \mathbb{D}),\]
then
\[\left|\frac{zf'(z)}{f(z)}-\alpha\right| <\alpha\quad (z\in \mathbb{D}),\]
where $\alpha$ is given by
\begin{equation}\label{sec3,eq1}
\alpha:=\frac{3(|b|+6)+\sqrt{9|b|^{2}+28|b|+4}}{8(|b|+4)}.
\end{equation}
In particular,
\[0<\RE \frac{zf'(z)}{f(z)}<2\alpha\quad (z\in \mathbb{D}).\]
\end{thm}

\begin{proof}
The hypothesis can be written in terms of subordination as
\[\frac{z f''(z)}{f'(z)}+1 \prec \frac{1-2z}{1-z}  \quad ( z\in \mathbb{D})\]
which gives $|b| \leq 1/2$. Also the constant $\alpha$ given by \eqref{sec3,eq1} satisfies the equation
\begin{equation}\label{sec3,eq2}
4(|b|+4)\alpha^{2}-3(|b|+6)\alpha+5=0.
\end{equation}
If $\alpha >2/3$, then we obtain
\[3\sqrt{9|b|^{2}+28|b|+4}>7|b|+10.\]
On solving this, we get $|b|>1/2$ which is a contradiction. Similarly, if we let $\alpha <5/8$, then we obtain $|b|<0$ which is again a contradiction. Thus $\alpha \in [5/8,2/3]$.

Define the function
\[w=q(z):=\frac{\alpha(1-z)}{(\alpha-1)z+\alpha}\quad (z\in {\mathbb{D}}) \]
where $\alpha$ is given by \eqref{sec3,eq1}. As $\alpha \in [5/8,2/3]$, $q$ is analytic   and univalent in $\overline{\mathbb{D}}$. Thus, $q\in {Q}$. Since $q(-1)=2\alpha$ and $q(1)=0$, we see that
\[q(\mathbb{D})=\{w:|w-\alpha|<\alpha\}.\]
Now, define the function $p:\mathbb{D} \rightarrow  \mathbb{C}$ by
\[p(z):=\frac{zf'(z)}{f(z)}.\]
Since $f \in \mathcal{A}_{b}$ and $f$ is starlike (univalent), the function
\[p(z)=1+bz+\cdots\]
is analytic in $\mathbb{D}$. Thus $p \in \mathcal{H}_{b}[1,1]$ and satisfies
\begin{equation}\label{sec3,eq3}
\frac{zf''(z)}{f'(z)}+1=p(z)+\frac{zp'(z)}{p(z)}=\psi(p(z),zp'(z))\quad (z \in \mathbb{D})
\end{equation}
where
\[\psi(r,s):=r+\frac{s}{r}.\]
We claim that $\psi \in \Psi_{b}(\Omega,q)$ where $\Omega=\{w:\RE w<3/2\}$.
The function $\psi$ is continuous in the domain $D=(\mathbb{C} \backslash {\{0\}}) \times \mathbb{C}$, $(1,0)\in {D}$ and
\[\RE  \psi (1,0)=1<3/2\]
so that $\psi (1,0)\in {\Omega}$. We now show that
\[ \RE \psi (q(\zeta),m \zeta q'(\zeta)) \geq \frac{3}{2} , \]
where $|\zeta|=1$ and
\[m \geq 1+\frac{|q'(0)|-|b|}{|q'(0)|+|b|},\quad q'(0)=\frac{1-2\alpha}{\alpha}.\]
Since
  \begin{align*}
                 \psi (q(\zeta),m \zeta q'(\zeta))&=q(\zeta) + m  \frac{\zeta q'(\zeta)}{q(\zeta)}\\
                                                  &=\frac{\alpha(1-\zeta)}{(\alpha-1)\zeta+\alpha}+\frac{m(1-2\alpha)\zeta}{(1-\zeta)[(\alpha-1)\zeta+\alpha]}\\
                                                  &=-1+\frac{(m+2)\alpha-\zeta}{(\alpha-1)\zeta+\alpha}-\frac{m}{1-\zeta},\quad \zeta\neq 1,
  \end{align*}
  we have
  \begin{equation}\label{sec3,eq4}
   \RE \psi (q(\zeta),m \zeta q'(\zeta))= -1+\RE\left(\frac{(m+2)\alpha-\zeta}{(\alpha-1)\zeta+\alpha}\right)-m\RE\frac{1}{1-\zeta}, \quad \zeta \neq 1.
  \end{equation}
Since for  $\alpha \in [5/8,2/3]$, $m\geq 1$,
\[\RE\left(\frac{(m+2)\alpha-\zeta}{(\alpha-1)\zeta+\alpha}\right)\geq (m+2)\alpha +1,\quad  |\zeta|=1,\]
and
\[\RE \frac{1}{1-\zeta}=\frac{1}{2},\quad |\zeta|=1,\quad \zeta \neq 1,\]
we have
  \begin{align*}
    \RE   \psi (q(\zeta),m \zeta q'(\zeta))&\geq -1+\frac{2(m+2)\alpha^{2}-m\alpha-1}{2\alpha-1}-\frac{m}{2}\\
                                           &=(m+2)\alpha-\frac{m}{2}\\
                                           &=\left(\frac{2\alpha-1}{2}\right)m+2\alpha\\
                                           &\geq \left(\frac{2\alpha-1}{2}\right)\left(1+\frac{(2\alpha-1)-|b|\alpha}{(2\alpha-1)+|b|\alpha}\right)+2\alpha\\
                                           &=\frac{2(|b|+4)\alpha^{2}-6\alpha+1}{(2\alpha-1)+\alpha|b|}=\frac{3}{2},
  \end{align*}
  using \eqref{sec3,eq2}.
Thus, $\psi \in \Psi_{b} (\Omega,q)$ where $\Omega=\{w:\RE w<3/2\}$.

From the hypothesis and \eqref{sec3,eq3}, we obtain
\[\psi(p(z),zp'(z))\in {\Omega}\quad (z \in \mathbb{D}).\]
Therefore, by applying Theorem \ref{sec1,th3}, we have
\[p(z)\prec q(z)\quad(z \in \mathbb{D})\]
or equivalently
\[ \left|\frac{zf'(z)}{f(z)}-\alpha\right| <  \alpha \quad (z \in \mathbb{D}).\]
In particular, the above inequality yields the following:
\[ 0<\RE \frac{zf'(z)}{f(z)}<2\alpha \quad (z \in \mathbb{D}).\qedhere\]
\end{proof}

\begin{rem}\label{sec2,rem12}
If $|b|=1/2$ then $\alpha$ given by \eqref{sec3,eq1} simplifies to $2/3$. Thus Theorem \ref{sec3,th1} reduces to \cite[Main theorem]{nunokawa} in this case.
\end{rem}

Another familiar implication is the following  \cite[Theorem 2.6i, p.\ 68]{monograph}:
\[\RE \left(\frac{z f''(z)}{f'(z)}+1\right) >-\frac{1}{2}\quad (z \in \mathbb{D}) \quad \Rightarrow \quad \RE \frac{z f'(z)}{f(z)}>\frac{1}{2}\quad (z \in \mathbb{D}) \]
for a function $f \in \mathcal{A}_{2}$. We generalize this result for a function $f \in \mathcal{A}_{2,b}$.

\begin{thm}\label{sec3,th3}
If $f \in \mathcal{A}_{2,b}$, then the following implication holds:
\[\RE \left(\frac{z f''(z)}{f'(z)}+1\right) >-\frac{1}{2} \quad (z \in \mathbb{D})\quad \Rightarrow \quad \RE \frac{z f'(z)}{f(z)}>\alpha\quad (z \in \mathbb{D}),\]
where $\alpha$ is the smallest positive root of the equation
\begin{equation}\label{sec3,eq5}
2\alpha^{3}+2(1-|b|)\alpha^{2}-(2|b|+7)\alpha+3+|b|=0
\end{equation}
in the interval $[1/2,2/3]$.
\end{thm}

\begin{proof}First note that in terms of subordination the hypothesis can be written as
\[\frac{z f''(z)}{f'(z)}+1 \prec \frac{1+2z}{1-z} \quad (z\in \mathbb{D})\]
which gives $|b| \leq 1/2$. Also, the function $g$ defined by
\[g(\alpha):=2\alpha^{3}+2(1-|b|)\alpha^{2}-(2|b|+7)\alpha+3+|b|\]
is continuous in $[1/2,2/3]$ and satisfies
\begin{align*}
    g\left(\frac{1}{2}\right)&=\frac{1}{4}(1-2|b|) \geq 0,\quad \mbox{and}\\
    g\left(\frac{2}{3}\right)&=-\frac{1}{27}(5+33|b|) \leq 0,
\end{align*}
as $|b| \leq 1/2$.
Therefore by Intermediate Value Theorem, there exists a root of $g(\alpha)=0$ in $[1/2,2/3]$. (In fact, $\alpha \in [0.5,\sqrt{2.5}-1]\simeq[0.5,0.58]$.)

Define the function $p:\mathbb{D} \rightarrow  \mathbb{C}$ by
\begin{equation}\label{sec3,eq6}
p(z):=\frac{zf'(z)}{f(z)}-\alpha,
\end{equation}
where $\alpha$ is the smallest positive root of \eqref{sec3,eq5}.
Since $f \in \mathcal{A}_{2,b}$ and $f$ is univalent, the function
\[p(z)=(1-\alpha)+2bz^{2}+\cdots\]
is analytic in $\mathbb{D}$. Thus $p \in \mathcal{H}_{2b}[1-\alpha,2]$ and as $\alpha \leq 2/3$ we readily see that
\[\RE p(0)=1-\alpha >0.\]
From \eqref{sec3,eq6}, we obtain
\[\frac{zf'(z)}{f(z)}=p(z)+\alpha\]
so that
\[\frac{zf''(z)}{f'(z)}+1=p(z)+\alpha+\frac{zp'(z)}{p(z)+\alpha}=\psi(p(z),zp'(z))\quad (z \in \mathbb{D})\]
where
\[\psi(r,s):=r+\alpha+\frac{s}{r+\alpha}.\]
We need to apply Theorem \ref{sec1,th4} to conclude that $\RE p(z) >0$. If we let
\[\Omega=\{w:\RE w >-1/2\}\]
then by hypothesis, we have
\[\{\psi(p(z),zp'(z)):z \in \mathbb{D}\} \subset \Omega.\]
To apply Theorem \ref{sec1,th4}, we need to show that $\psi \in \Psi_{2,2b}(\Omega,1-\alpha)$. The function $\psi$ is continuous in the domain $D=(\mathbb{C}\backslash \{-\alpha\})\times \mathbb{C}$, $(1-\alpha,0) \in D$ and
\[\RE \psi(1-\alpha,0)=1 >0.\]
We now show that the admissibility condition \eqref{sec1,eq1} is satisfied. Since
\[\psi(i\rho,\sigma)=i \rho+\alpha+\frac{\sigma}{\alpha^{2}+\rho^{2}}(\alpha-i\rho)\]
we have
\begin{align*}
\RE \psi(i\rho,\sigma)&=\alpha+\frac{\alpha \sigma}{\alpha^{2}+\rho^{2}}\\
                        &\leq \alpha-\frac{1}{2}\frac{\alpha}{\alpha^{2}+\rho^{2}}\left(2+\frac{2(1-\alpha)-2|b|}{2(1-\alpha)+2|b|}\right)\frac{(1-\alpha)^{2}+\rho^{2}}{1-\alpha}\\
                        &=\alpha-\frac{1}{2}\frac{\alpha}{1-\alpha}\frac{3(1-\alpha)+|b|}{(1-\alpha)+|b|}\frac{(1-\alpha)^{2}+\rho^{2}}{\alpha^{2}+\rho^{2}}.
\end{align*}
Using \eqref{sec3,eq5} and from the monotonicity of the function
\[h(t)=\frac{(1-\alpha)^{2}+t}{\alpha^{2}+t},\quad t \geq0,\]
it follows that
\begin{align*}
\RE \psi(i\rho,\sigma)&\leq \alpha-\frac{1}{2}\frac{1-\alpha}{\alpha}\frac{3(1-\alpha)+|b|}{(1-\alpha)+|b|}\\
                        &=\frac{(2|b|-1)\alpha^{2}-2\alpha^{3}+(6+|b|)\alpha-3-|b|}{2\alpha[(1-\alpha)+|b|]}=-\frac{1}{2}.
  \end{align*}
whenever $\rho \in \mathbb{R}$ and
\[\sigma \leq -\frac{1}{2}\left(2+\frac{2\RE p(0)-\beta}{2\RE p(0)+\beta}\right)\frac{|p(0)-i\rho|^{2}}{\RE p(0)},\quad p(0)=1-\alpha, \quad \beta=2|b|.\]
Thus $\psi \in \Psi_{2,2b}(\Omega,1-\alpha)$. Therefore, by applying part $(i)$ of Theorem \ref{sec1,th4} we conclude that $p$ satisfies
\[\RE p(z) >0 \quad ( z \in \mathbb{D}).\]
This is equivalent to
\[\RE \frac{z f'(z)}{f(z)}>\alpha\quad (z \in \mathbb{D}).\]
where $\alpha$ is the smallest positive root of \eqref{sec3,eq5}.
\end{proof}

\begin{rem}\label{sec3,rem4}
If $|b|=1/2$ then \eqref{sec3,eq5} becomes
\[2\alpha^{3}+\alpha^{2}-8\alpha+\frac{7}{2}=0\]
which simplifies to
\[(2\alpha-1)\left(\alpha^{2}+\alpha-\frac{7}{2}\right)=0.\]
As $\alpha \in [1/2,2/3]$, we get $\alpha=1/2$. Thus Theorem \ref{sec3,th3} reduces to Theorem 2.6i in \cite[p.\ 68]{monograph} in this case.
\end{rem}

\subsection*{Acknowledgements}
The research work of the first author is supported by research fellowship from Council of Scientific and Industrial Research (CSIR), New Delhi. The authors are thankful to the referee for his useful comments.

\end{document}